\documentclass[10pt,letterpaper,fleqn,oneside]{article}
\pdfoutput=1

\usepackage{sectsty}
\usepackage{caption}
\usepackage{times}
\usepackage{amssymb,amsfonts,amsmath,amscd,amsthm}
\usepackage[pdftex,colorlinks]{hyperref}
\usepackage[pdftex]{graphicx}
\usepackage{fnpos}
\usepackage{subfigure}
\usepackage{verbatim}
\usepackage{fancyhdr}
\usepackage[numbers]{natbib}

\usepackage{url}

\usepackage[usenames,dvipsnames]{color}
\usepackage{soul}   

\newtheorem{lemma}{Lemma}

\DeclareMathOperator*{\diag}{diag}
\newcommand{\bbm}{\begin{bmatrix}}
\newcommand{\ebm}{\end{bmatrix}}
\newcommand{\norm}[2]{\left|\left| #1 \right|\right|_{#2}}

\newcommand{\uniform}[1]{\mathcal{U}\left(#1\right)}
\newcommand{\unitBall}[1]{\zeta_{#1}}
\newcommand{\unitNBall}[0]{\unitBall{n}}

\newcommand{\xFirstFoci}[0]{\statex_{\rm f1}}    
\newcommand{\xSecondFoci}[0]{\statex_{\rm f2}}    

\newcommand{\statex}[0]{\mathbf{x}}
\newcommand{\namedState}[1]{\statex_{#1}}
\newcommand{\stateSet}[0]{X}
\newcommand{\namedSet}[1]{\stateSet_{#1}}
\newcommand{\namedPdf}[2]{p_{#1}\left(#2\right)}

\newcommand{\xball}[0]{\namedState{\rm ball}}
\newcommand{\xcentre}[0]{\namedState{\rm centre}}
\newcommand{\xellipse}[0]{\namedState{\rm ellipse}}
\newcommand{\ballSet}[0]{\namedSet{\rm ball}}
\newcommand{\ellipseSet}[0]{\namedSet{\rm ellipse}}
\newcommand{\ballPdf}[1]{\namedPdf{\rm ball}{#1}}
\newcommand{\ellipsePdf}[1]{\namedPdf{\rm ellipse}{#1}}

\makeFNbottom \makeFNbelow

\newcommand{\UTIASprogram}{N/A}
\newcommand{\UTIAStitle}{The Probability Density Function of a Transformation-based Hyperellipsoid Sampling Technique}
\newcommand{\UTIASdocument}{TR-2014-JDG004}
\newcommand{\UTIASrevision}{Rev: 1.2}
\newcommand{\UTIASauthor}{J.\ D.\ Gammell}

\hypersetup{%
    pdftitle={\UTIASdocument: \UTIAStitle},
    pdfauthor={\UTIASauthor},
    pdfkeywords={},
    pdfsubject={\UTIASprogram},
    pdfstartview=FitH,%
    bookmarksopen=true,%
    breaklinks=true,%
    colorlinks=true,%
    linkcolor=blue,anchorcolor=blue,%
    citecolor=blue,filecolor=blue,%
    menucolor=blue,
    urlcolor=blue
}%

\title{\sf\bfseries \UTIAStitle }

\author{Jonathan D.\ Gammell\\
       Institute for Aerospace Studies \\
       University of Toronto\\
       4925 Dufferin Street \\
       Toronto, Ontario \\
       Canada M3H 5T6 \\
       \texttt{<jon.gammell@utoronto.ca>}
       \and
       Timothy D.\ Barfoot \\
       Institute for Aerospace Studies \\
       University of Toronto\\
       4925 Dufferin Street \\
       Toronto, Ontario \\
       Canada M3H 5T6 \\
       \texttt{<tim.barfoot@utoronto.ca>}}

\date{}

\graphicspath{{figs/}}

\begin{document}

\fancypagestyle{plain}{%
    \fancyhf{}%
    \fancyfoot[C]{}%
    \fancyhead[R]{\begin{tabular}[b]{r}\small\sf \UTIASdocument\\
         \small\sf\UTIASrevision\\
         \small\sf\today \end{tabular}}%
    \renewcommand{\headrulewidth}{0pt}
    \renewcommand{\footrulewidth}{0pt}%
}

\pagestyle{fancy}
\allsectionsfont{\sf\bfseries}
\renewcommand{\captionlabelfont}{\sf\bfseries}
\reversemarginpar
\renewcommand{\labelitemi}{--}
\renewcommand{\labelitemii}{--}
\renewcommand{\labelitemiii}{--}


\rhead{ \begin{tabular}[b]{r}\small\sf \UTIASdocument\\
        \small\sf\UTIASrevision\\
        \small\sf\today \end{tabular}}
\chead{}
\lfoot{}
\cfoot{\thepage}
\rfoot{}
\renewcommand{\headrulewidth}{0pt}
\renewcommand{\footrulewidth}{0pt}


\maketitle

\begin{abstract}
Sun and Farooq \cite{sun_fusion02} showed that random samples can be efficiently drawn from an arbitrary $n$-dimensional hyperellipsoid by transforming samples drawn randomly from the unit $n$-ball. They stated that it was a \textit{straightforward} to show that, given a uniform distribution over the $n$-ball, the transformation results in a uniform distribution over the hyperellipsoid, but did not present a full proof. This technical note presents such a proof.
\end{abstract}

\section{Transformation-based Sampling of Hyperellipsoids}\label{sec:sample}
Let $\ellipseSet$ be the set of points within an $n$-dimensional hyperellipsoid such that
\begin{align*}
    \ellipseSet = \left\lbrace \statex \in \mathbb{R}^n \;\; \middle| \;\;  \left( \statex - \xcentre \right)^T\mathbf{S}^{-1}\left( \statex - \xcentre \right) \leq 1 \right\rbrace,
\end{align*}
where $\mathbf{S} \in \mathbb{R}^{n \times n}$ is the hyperellipsoid matrix, and $\xcentre = \left(\xFirstFoci + \xSecondFoci\right)/2$ is the centre of the hyperellipsoid in terms of its two focal points, $\xFirstFoci$ and $\xSecondFoci$.
We can then transform points from the unit $n$-ball, $\xball \in \ballSet$, to points in the hyperellipsoid, $\xellipse \in \ellipseSet$, by a linear invertible transformation as,
\begin{align} \label{eqn:transform}
    \xellipse = \mathbf{L} \xball + \xcentre.
\end{align}
The transformation $\mathbf{L}$ is given by the Cholesky decomposition of the hyperellipsoid matrix,
\begin{align*}
    \mathbf{L}\mathbf{L}^T \equiv \mathbf{S},
\end{align*}
and the unit $n$-ball is defined in terms of the Euclidean norm, $\norm{\cdot}{2}$, by
\begin{align*}
    \ballSet = \left\lbrace \statex \in \mathbb{R}^n \;\; \middle| \;\;  \norm{\statex}{2} \leq 1 \right\rbrace.
\end{align*}

\section{Resulting Probability Density Function}
In response to concerns expressed by Li \cite{li_cntrlapp92} that sampling the hyperellipsoid by transforming uniformly-drawn samples from the unit $n$-ball, $\xball \sim \uniform{\ballSet}$, by \eqref{eqn:transform} would not result in a uniform distribution, Sun and Farooq \cite{sun_fusion02} stated the following Lemma and Proof.
\begin{lemma}
If the random points distributed in a hyper-ellipsoid are generated from the random points uniformly distributed in a hyper-sphere through a linear invertible non-orthogonal transformation, then the random points distributed in the hyper-ellipsoid are also uniformly distributed.
\end{lemma}
\begin{proof}
The proof of the above lemma is very straightforward and is omitted here for brevity. The result of the lemma is further substantiated through the simulation shown in [Figures].
\end{proof}

For clarity, the full proof is presented below.
\begin{proof}
Let $\ballPdf{\cdot}$ be the probability density function of samples drawn uniformly from the unit $n$-ball of volume $\unitNBall$, such that,
\begin{align}\label{eqn:ballPdf}
    \ballPdf{\statex} :=
    \begin{cases}
        \dfrac{1}{\unitNBall},& \forall \statex \in \ballSet\\
        0,              & \text{otherwise},
    \end{cases}
\end{align}
and $g\left(\cdot\right)$ be an invertible transformation from the unit $n$-ball to a hyperellipsoid, such that,
\begin{align*}
    \xellipse &:= g\left(\xball\right),\\
    \xball &= g^{-1}\left(\xellipse\right).
\end{align*}
Then the probability density function of samples drawn from the hyperellipsoid, $\ellipsePdf{\cdot}$, is given by,
\begin{align}\label{eqn:pdfDefn}
    \ellipsePdf{\statex} := \ballPdf{g^{-1}\left(\statex\right)}\left|\det\left\lbrace \left.\frac{dg^{-1}}{d\xellipse}\right|_{\statex} \right\rbrace \right|.
\end{align}
From \eqref{eqn:transform}, we can calculate the inverse transformation as,
\begin{align*}
    g^{-1}\left(\xellipse\right) = \mathbf{L}^{-1}\left(\xellipse - \xcentre\right),
\end{align*}
whose Jacobian is then
\begin{align}\label{eqn:jac}
    \frac{dg^{-1}}{d\xellipse} = \frac{d}{d\xellipse}\mathbf{L}^{-1}\left(\xellipse - \xcentre\right) = \mathbf{L}^{-1}.
\end{align}
Substituting \eqref{eqn:jac} and \eqref{eqn:ballPdf} into \eqref{eqn:pdfDefn} gives,
\begin{align}\label{eqn:ellipsePdf}
    \ellipsePdf{\statex} :=
    \begin{cases}
        \dfrac{1}{\unitNBall}\left|\det\left\lbrace \mathbf{L}^{-1} \right\rbrace \right|,& \forall \statex \in \ellipseSet\\
         0,              & \text{otherwise},
    \end{cases}
\end{align}
where we have used the fact that $g^{-1}\left(\statex\right) \in \ballSet \implies \statex \in \ellipseSet$.
As $\ellipsePdf{\cdot}$ is constant for all $\xellipse \in \ellipseSet$, this proves that \eqref{eqn:transform} transforms samples drawn uniformly from the unit $n$-ball such that they are uniformly distributed over the hyperellipsoid given by $\mathbf{S}$.
\end{proof}

\subsection{Orthogonal Hyperellipsoids}
If the axes of hyperellipsoid are orthogonal, there is a coordinate frame aligned to the axes of the hyperellipsoid such that $\mathbf{S}$ will be diagonal,
\begin{align*}
    \mathbf{S}' = \diag\left\lbrace r_1^2, r_2^2, \ldots, r_n^2 \right\rbrace,
\end{align*}
where $r_i$ is the radius of $i$-th axis of the hyperellipsoid.
The transformation from the unit $n$-ball to the hyperellipsoid expressed in this aligned frame, $\mathbf{L}'$, will then be
\begin{align}\label{eqn:orthoL}
    \mathbf{L}' = \diag\left\lbrace r_1, r_2, \ldots, r_n \right\rbrace.
\end{align}
The hyperellipsoid in any arbitrary Cartesian frame can then be expressed as a rotation applied after this diagonal transformation,
\begin{align}\label{eqn:rotC}
    \xellipse = \mathbf{C}\mathbf{L'} \xball + \xcentre,
\end{align}
where $\mathbf{C} \in SO\left(n\right)$ is an $n$-dimensional rotation matrix.
Rearranging \eqref{eqn:rotC} and substituting into \eqref{eqn:ellipsePdf} gives
\begin{align}\label{eqn:orthoPdfTemp}
    \ellipsePdf{\statex} :=
    \begin{cases}
        \dfrac{1}{\unitNBall}\left|\det\left\lbrace \mathbf{L'}^{-1}\mathbf{C}^T \right\rbrace \right|,& \forall \statex \in \ellipseSet\\
         0,              & \text{otherwise},
    \end{cases}
\end{align}
where we have made use of the orthogonality of rotation matrices, $\forall \mathbf{C} \in SO\left(n\right), \, \mathbf{C}^T \equiv \mathbf{C}^{-1}$.
Substituting \eqref{eqn:orthoL} into \eqref{eqn:orthoPdfTemp} finally gives,
\begin{align}\label{orthoPdf}
        \ellipsePdf{\statex} :=
    \begin{cases}
        \dfrac{1}{\unitNBall\prod_{i=1}^n r_i},& \forall \statex \in \ellipseSet\\
         0,              & \text{otherwise},
    \end{cases}
\end{align}
Where we have made use of the fact that all rotation matrices have a unity determinant, $\forall \mathbf{C} \in SO\left(n\right), \, \det\left\lbrace\mathbf{C}\right\rbrace = 1,$ and that the determinant of a diagonal matrix is the product of the diagonal terms. As expected, \eqref{orthoPdf} is exactly the inverse of the volume of an $n$-dimensional hyperellipsoid with radii $\left\lbrace r_i \right\rbrace$.
\newpage
\bibliographystyle{asrl}
\bibliography{TR-2014-JDG004}

\end{document}